\documentclass{amsart}
\usepackage{amsfonts,amssymb,amsmath,amsthm}
\usepackage{url}
\usepackage{enumerate}

\newtheorem{thm}{Theorem}[section]
 \newtheorem{cor}[thm]{Corollary}
 \newtheorem{lem}[thm]{Lemma}
   \newtheorem{examp}[thm]{Example}
  \newtheorem{defin}[thm]{Definition}
  \newtheorem{prop}[thm]{Proposition}
  \newtheorem{rem}[thm]{Remark}
\author{Mahboubeh Sanaei $^{*}$, 
Shervin Sahebi $^{**}$ and Hamid H. S. Javadi $^{***}$}
\address{$^{*,**}$ Department of Mathematics, Islamic Azad University,
Central Tehran Branch, 13185/768, Iran. 
$^{***}$ Department of Mathematics and Computer Science, Shahed University, Tehran, Iran.
}
\email{sahebi@iauctb.ac.ir; mah.sanaei.sci@iauctb.ac.ir; h.s.javadi@shahed.ac.ir.
}

\begin{document}

\title[On Central Skew Armendariz Rings]{On Central Skew Armendariz Rings}

\begin{abstract}
For a ring endomirphism $ \alpha $, we introduce the central $ \alpha $-skew Armendariz rings, which are a generalization of $\alpha$-skew Armendariz rings and central Armendariz rings, and investigate their properties. For a ring $R$, we show that if $ \alpha(e)=e $ for each idempotent $ e\in R $, then $R$ is a central $ \alpha $-skew Armendariz ring if and only if $R$ is abelian; and $eR$ and $(1-e)R$ are central $\alpha$-skew Armendariz for some $e^{2}=e\in R$. We prove that if $\alpha^{t}=I_{R}$ for some positive integer $t$, $ R $ is central $ \alpha $-skew Armendariz if and only if the polynomial ring $ R[x] $ is central $ \alpha $-skew Armendariz if and only if the Laurent polynomial ring $R[x,x^{-1}]$ is central $\alpha$-skew Armendariz. Moreover, it is proven that if $\alpha(e)=e$ for any $e^{2}=e \in R$, and $R$ is a central $\alpha$-skew Armendariz ring, then $R$ is right p.p-ring if and only if $R[x;\alpha]$ is right p.p-ring. 

\end{abstract}
\maketitle

\section{Introduction} \label{sect1}
Throughout this article, $ R $ denotes an associative ring with identity. The center of a ring $R$ and the set of all the units in $ R $ are denoted by $C(R)$ and $ U(R) $ repectively. In 1997, Rege and Chhawchharia introduced the notion of an Armendariz ring. They called a ring $ R $ an \textit{Armendariz ring} if whenever polynomials $ f(x)=a_{0}+a_{1}x+\cdots +a_{n}x^{n}$ and $ g(x)=b_{0}+b_{1}x+\cdots +b_{m}x^{m} \in R[x]$ satisfy $ f(x)g(x)=0 $ then $ a_{i}b_{j} =0$ for all $ i$ and $j $. The name "Armendariz ring" is chosen because Armendariz \cite[Lemma 1]{EP} has been shown that reduced ring (that is a ring without nonzero nilpotent) saisfies this condition. A number of properties of the Armendariz rings have been studied in 
\cite{a02, EP, a03, a04, Reg} 
So far Armendariz rings are generalized in several forms.
Let $ \alpha $ be an endomorphism of a ring $ R $. Hong et al., (2003) \cite{Rizvi} give a possible generalization of the Armendariz rings. A ring $ R $ is called $ \alpha $-\textit{Armendariz} if for any $ f(x)= \sum_{i=0}^{n}a_{i}x^{i},g(x)= \sum_{j=0}^{m}b_{j}x^{j} \in R[x;\alpha];~ f(x)g(x)=0 $ implies that $ a_{i}b_{j}=0 $ for all $ i,j $. According to \cite{alpha}, a ring $ R $ is called $ \alpha $-\textit{skew Armendariz} if $ f(x)g(x)=0 $, such that $ f(x)= \sum_{i=0}^{n}a_{i}x^{i},g(x)= \sum_{j=0}^{m}b_{j}x^{j} \in R[x;\alpha] $, implies that $ a_{i}\alpha^{i}(b_{j})=0 $ for all $ i,j $.
They showed that if a ring $R$
is 
$\alpha$-rigid (that is, if $a\alpha(a)=0$ then $a=0$ for $a\in R$), then $R[x]/\langle x^{2}\rangle$ is $\bar{\alpha}$-skew Armendariz. They also showed that if $\alpha^{t}=I_{R}$ for some positive integer $t$, then $R$ is $\alpha$-skew Armendariz if and only if $R[x] $ is $\alpha$-skew Armendariz. Agayev et al., \cite{a01} called a ring $R$ central Armendariz if whenever polynomials 
$f(x)=a_{0}+a_{1}x+\cdots + a_{n}x^{n}$,
$g(x)=b_{0}+b_{1}x+\cdots + b_{m}x^{m} \in R[x]$ satisfy $f(x)g(x)=0$, then $a_{i}b_{j}\in C(R)$ for all $i$ and $j$. They showed that the class of central Armendariz rings lies precisely between classes of Armendariz rings and abelian rings (that is, its idempotents belong to $C(R)$.) For a ring $R$, they proved that $R$ is central Armendariz if and only if $R[x]$ is central Armendariz if and only if
$R[x,x^{-1}]$ is central Armendariz where $ R[x] $ is the polynomial ring and $ R[x,x^{-1}] $ is the Laurent polynomial ring over a ring $ R $. Furthermore, they showed that if $R$ is reduced, consequently $R[x]/\langle x^{n}\rangle$ is central Armendariz, and the converse holds if $R$ is semiprime, where $\langle x^{n}\rangle$ is the ideal generated by $x^{n}$ and $n\geqslant 2$.
Motivated by the above results, for an endomorphism $\alpha$ of a ring $R$, we investigate a generalization of the $\alpha$-skew Armendariz rings and the central Armendariz rings.

\section {Central $ \alpha $-skew Armendariz rings}

 \noindent
In this section, the central $ \alpha $-skew Armendariz rings are introduced 
as a generalization of the $ \alpha $-skew Armendariz ring. 
\begin{defin}\label{def 1.1}  
Let $ \alpha $ be an endomorphism of a ring $ R $. The ring $ R $ is called a central $ \alpha $-skew Armendariz ring if for any nonzero polynomial  $ f(x)=\sum_{i=0}^{n}a_{i}x^{i}$ and $g(x)= \sum_{j=0}^{m}b_{j}x^{j} \in R[x;\alpha] ,  $ $ f(x)g(x)=0 $, implies that
 $a_{i}\alpha^{i}(b_{j}) \in C(R)$ for each $i,~j $.
\end{defin}
Note that all commutative rings, $ \alpha $-skew Armendariz rings and the subrings of central $ \alpha $-skew Armendariz rings are central $ \alpha $-skew Armendariz. Also, since each reduced ring $ R $ is $ I_{R} $-skew Armendariz where $ I_{R} $ is an identity map; each reduced ring is central $ I_{R} $-skew Armendariz ring. \\ 
The following examples show that the central $\alpha$-skew Armendariz rings are not nescessary $\alpha$-skew Armendariz.
 \begin{examp}\label{examp1}(1)
Let $ R=R_{1}\oplus R_{2} $, where $ R_{i} $ is a commutative ring for $ i=1,2 $. Let $ \alpha :R\longrightarrow R $ be an automorphism defined by $ \alpha ((a,b))= (b,a) $, then for $ f(x)=(1,0)-(1,0)x $ and $ g(x)= (0,1)+(1,0)x $ in $ R[x;\alpha] $, $ f(x)g(x)=0 $, but $ (1,0)\alpha((0,1))=(1,0)^{2}\neq 0 $. Therefore, $ R $ is not $ \alpha  $-skew Armendariz. But  $ R $ is central $ \alpha $-skew Armendariz, since $ R $ is commutative.\\
(2)
Let $ \mathbb{Z}_{4} $ be the ring of integers modulo $ 4 $. Consider a ring
$
R=\Big\{ 
\bigl( \begin{smallmatrix}
  \bar{a}&\bar{b}\\ 0&\bar{a}
\end{smallmatrix} \bigr) \Big| \bar{a},\bar{b} \in \mathbb{Z}_{4} \Big\}
$
and $ \alpha : R\longrightarrow R $ be an endomorphism defined by 
$ \alpha \Big( \bigl( \begin{smallmatrix}
  \bar{a}&\bar{b}\\ 0&\bar{a}
\end{smallmatrix} \bigr) \Big) = \bigl( \begin{smallmatrix}
  \bar{a}&- \bar{b}\\ 0&\bar{a}
\end{smallmatrix} \bigr) $.
The ring $ R $ is not $ \alpha $-skew Armendariz. In fact 
$ \Big( \bigl( \begin{smallmatrix}
  \bar{2}&\bar{0}\\ 0&\bar{2}
\end{smallmatrix} \bigr) + \bigl( \begin{smallmatrix}
  \bar{2}&\bar{1}\\ 0&\bar{2}
\end{smallmatrix} \bigr) x \Big) ^{2}=0 \in R[x;\alpha]$
but
$ \bigl( \begin{smallmatrix}
  \bar{2}&\bar{1}\\ 0&\bar{2}
\end{smallmatrix} \bigr)  \alpha \Big(\bigl( \begin{smallmatrix}
  \bar{2}&\bar{0}\\ 0&\bar{2}
\end{smallmatrix} \bigr) \Big)= \bigl( \begin{smallmatrix}
  \bar{0}&\bar{2}\\ 0&\bar{0}
\end{smallmatrix} \bigr) \neq 0$. But it can be easily checked that $ R $ is commutative and so, it is central $ \alpha $-skew Armendariz ring.
  
\end{examp}
Let 
$R_{i} $
be a ring and 
$\alpha_{i} $
an endomorphism of $R_{i}$, for each
$i\in I$. 
Then for  
$\prod_{i\in I}R_{i}$
and the endomorphism 
$\bar{\alpha}:\prod_{i\in I}R_{i}\longrightarrow \prod_{i\in I}R_{i}$
defined by
$\bar{\alpha}(a_{i})=(\alpha(a_{i}))$, 
$\prod_{i\in I}R_{i} $
is central 
$\alpha$-skew Armendariz if and only if each $R_{i}$
is central 
$\alpha$-skew Armendariz.

  \begin{prop}
Let $\alpha$ be an endomorphism of a ring $R$. Let $S$ be a ring and $\varphi : R\longrightarrow S$ an isomorphism. Then R is central  $\alpha$-skew Armendariz if and only if $S$ is central $\varphi\alpha \varphi^{-1}$ -skew Armendariz.
\end{prop}
\begin{proof}
Let $\alpha^{\prime}= \varphi \alpha \varphi^{-1}$. Clearly, $\alpha^{\prime}$ is an endomorphism of $S$. Suppose that $a^{\prime} = \varphi(a)$, for $a \in R$. Note that $\varphi (a\alpha^{i}(b))= a^{\prime}\varphi (\alpha^{i}(b))$ for all $a,b\in R$. Also $f(x)=\sum_{i=0}^{n}a_{i}x^{i}$ and $g(x)=\sum_{j=0}^{m}b_{j}x^{j}$ are nonzero in $R[x;\alpha]$ if and only if 
 $f^{\prime}(x)=\sum_{i=0}^{n}a^{\prime}_{i}x^{i}$ and $g^{\prime}(x)=\sum_{j=0}^{m}b^{\prime}_{j}x^{j}$ are nonzero in $S[x;\alpha^{\prime}]$. On the other hand 
$f(x)g(x)=0$
in
$R[x;\alpha]$
if and only if
$f^{\prime}(x)g^{\prime}(x)=0$
in
$S[x;\alpha^{\prime}]$. Also since $\varphi$ is an isomorphism,
$a^{\prime}_{i}(\varphi \alpha\varphi)^{i}b^{\prime}_{j}=a^{\prime}_{i}\varphi\alpha^{i}\phi^{-1}(b^{\prime}_{j})=\varphi(a_{i})\varphi\alpha^{i}(b_{j})=\varphi(a_{i}\alpha^{i}(b_{j})) \in C(S)$ if and only if 
$a_{i}\alpha^{i}(b_{j})\in C(R)$.
 Thus $R$ is central $\alpha$-skew Armendariz if and only if $S$ is central $\varphi\alpha\varphi^{-1}$-skew Armendariz.
\end{proof}

The following example shows that there exists a central $\alpha$-skew Armendariz ring such that 
$\alpha(e)\neq e$
for some $e^{2}=e \in R$.
\begin{examp}
Let
$R=\Big\{ 
\bigl( \begin{smallmatrix}
  a & 0\\ 0& b
\end{smallmatrix} \bigr) \Big| a,b \in \mathbb{Z}_{2}\oplus \mathbb{Z}_{2} \Big\}$, where $\mathbb{Z}_{2}$ is the ring of integers modulo $2$.
Let 
$\alpha : R\longrightarrow R$ be defined by 
$\alpha \Big( \bigl( \begin{smallmatrix}
  a & 0\\ 0& b
\end{smallmatrix} \bigr) \Big) =\bigl( \begin{smallmatrix}
  b & 0\\ 0& a
\end{smallmatrix} \bigr) $. Then $R$ is a commutative ring, so it is central $\alpha$-skew Armendariz. But  
$\alpha (e) \neq e$ for $e=\bigl( \begin{smallmatrix}
  (1,0) & (0,0)\\ (0,0)& (0,1)
\end{smallmatrix} \bigr)$.
\end{examp}
Recall that every ring $ R $ is said to be \textit{abelian} if every idempotent of $ R $ is central.
\begin{prop}\label{prop1}
Let $R$ be a ring and $ \alpha $ be an endomorphism with $ \alpha(e)=e $ for any $ e^{2}=e\in R $. Then $R$ is central $ \alpha $-skew Armendariz ring if and only if $R$ is abelian and $eR$ and $(1-e)R$ are central $\alpha$-skew Armendariz for some $e^{2}=e\in R$.

\end{prop}
\begin{proof}
If $R $ is central $\alpha$-skew Armendariz, $eR$ and $(1-e)R$ are central $\alpha$-skew Armendariz since they are the invariant subrings of $R$. Now let $ e $ be an idempotent in $ R $. Consider $ f(x)= e-er(1-e)x $ and $ g(x)= (1-e)+er(1-e)x $. Since $ \alpha(e)=e $ we have $ f(x)g(x)=0 $. By hypothesis $ er(1-e) $ is central and so 
$ er(1-e)=0 $. Hence $ er=ere $  for each $ r\in R $. 
Similarly consider $ p(x)=(1-e)- (1-e) rex $ and $ q(x)= e+(1-e)rex $ in $ R[x;\alpha] $ for all $ r\in R $. Then $ p(x)q(x)=0 $. As before $ (1-e)re=0 $ and $ ere=re $ for all $ r\in R $. It follows that $ e $ is central element of $ R $, that is, $ R  $ is abelian. Conversely, let  $f(x)=\sum_{i=0}^{n}a_{i}x^{i}$ and $g(x)=\sum_{j=0}^{m}a_{j}x^{j}$ be nonzero polynomials in $R[x;\alpha] $ such that $f(x)g(x)=0$. Let 
$f_{1}(x)=ef(x) $, $g_{1}(x)= eg(x)$, $\quad f_{2}(x)=(1-e)f(x) $, $g_{2}(x)= (1-e)g(x)$.
Then 
$f_{1}(x)g_{1}(x)=0$ in $(eR)[x;\alpha ]$ and 
$f_{2}(x)g_{2}(x)=0$ in $((1-e)R)[x;\alpha ]$.
Since $eR$ and $(1-e)R$ are central $\alpha$-skew Armendariz, 
$ea_{i}e\alpha^{i}(b_{j})$ is central in $eR$ and $(1-e)a_{i}(1-e)\alpha^{i}(b_{j})$ is central in $(1-e)R$. Since $e$ and $1-e$ are central in 
$R$, $R=eR\oplus(1-e)R$ and so 
$a_{i}\alpha^{i}(b_{j})= ea_{i}\alpha^{i}(b_{j}) + (1-e)a_{i}\alpha^{i}(b_{j})$
is central in $R$ for all 
$0\leqslant i\leqslant n$, $0\leqslant j\leqslant m$. Therefore $R$ is central $\alpha $-skew Armendariz.
\end{proof}
\begin{cor}\label{cor1}
Let $ \alpha $ be an endomorphism of a ring $ R $ with $ \alpha (e)=e $ for any $ e^{2}=e \in R $, if $ R $  is $ \alpha $-skew Armendariz ring then $ R  $ is abelian.
\end{cor}
\begin{rem}
In  \cite[Example 2.2]{a01} shows that abelian rings need not to be central Armendariz in general. Clearly, for any ring $R$ and endomorphism $\alpha= I_{R}$ the abelian rings  in general are not central $\alpha$-skew Armendariz.
\end{rem}
\begin{lem}\label{Lem1}
Let $R$ be a central $\alpha$-skew Armendariz ring. If $e^{2}=e \in R[x;\alpha]$, where $e=e_{0}+e_{1}x+\cdots +e_{n}x^{n}$, then
$e_{i}\in C(R)$ for $i=1,2,\cdots , n$. Moreover if $\alpha(e)=e$, then $e=e_{0}$.
\end{lem}
\begin{proof}
Since
 $e(1-e)=0=(1-e)e$, we have $(e_{0}+e_{1}x+\cdots +e_{n}x^{n})(1-e_{0}+e_{1}x+\cdots +e_{n}x^{n})=0$ and
 $(1-e_{0}+e_{1}x+\cdots +e_{n}x^{n})(e_{0}+e_{1}x+\cdots +e_{n}x^{n})=0$. Since $R$ is a central $\alpha$-skew Armendariz ring, 
$e_{0}e_{i}\in C(R)$ 
and 
$(1-e_{0})e_{i}\in C(R)$ for $1\leqslant i\leqslant n$. Thus $e_{i}\in C(R) $ for $1\leqslant i\leqslant n$. Now let $R$ is a central $\alpha$-skew Armendariz ring and $\alpha(e)=e$. By Proposition \ref{prop1}, it follows that $R$ is abelian. The rest follows from Theorem 2.9 in \cite{a3} 
\end{proof}
\begin{prop}
Let $R$
be a central $\alpha$-skew Armendariz ring. Then 
$R[x;\alpha]$ is abelian if and only if
$\alpha(e)=e$ for any $e^{2}=e\in R$.
\end{prop}
\begin{proof}
Suppose that 
$R[x;\alpha]$
is abelian and
$e^{2}=e \in R[x;\alpha]$. Then $e$ is central and so
$ex=xe=\alpha(e)x$. Thus 
$\alpha(e)=e$.\\
Conversely, let
$\alpha(e)=e$ for any $e^{2}=e\in R$. Since $R$ is central $\alpha$-skew Armendariz by Proposition \ref{prop1}, $R$ is abelian. Now we suppose that 
$e^{2}=e\in R[x;\alpha]$. By Lemma \ref{Lem1}, $e$ is an idempotent in $R$. For any 
$p=a_{0}x^{k}+a_{1}x^{k+1}+\cdots +a_{m}x^{k+m} \in R[x;\alpha]$ where 
$k$ and $m$ are nonnegative integers,
$pe=(a_{0}x^{k}+a_{1}x^{k+1}+\cdots +a_{m}x^{k+m})e=a_{0}\alpha^{k}(e) x^{k}+a_{1}\alpha^{k+}(e)x^{k+1}+\cdots +a_{m}\alpha^{k+m}(e)x^{k+m}=e(a_{0}x^{k}+a_{1}x^{k+1}+\cdots +a_{m}x^{k+m})=ep$,
 since $R$ is abelian and $\alpha(e)=e$. Therefore $R[x;\alpha]$  is abelian.     
\end{proof}
For a nonempty subset $X$ of a ring $R$, we write $r_{R}(X)=\lbrace  r\in R | xr=0 ~ for ~ any~ x\in X \rbrace$ which is called right annihilator of $X$ in $R$.
Kaplansky \cite{a01} introduced Baer rings as rings in which the right annihilator of every nonempty subset is generated by an idempotent. 
As a generalization of Baer rings, a ring $R$ is called a right (resp., left) $p.p$-ring if the right (resp., left) annihilator of an element of $R$ is generated (as  a right (resp., left) ideal) by an idempotent of $R$.
 \begin{thm}\label{THM1}
 For an endomorphism $ \alpha $ of a ring $ R $, if the ring $ R $ is $ \alpha $-skew Armendariz, then $R$ is central $\alpha$-skew Armendariz. The converse hold if $R$ is right p.p-ring and  $ \alpha(e)=e $ for any $ e^{2}=e \in R $. 
 \end{thm}
 \begin{proof}
 Suppose $R$ is a central $\alpha$-skew Armendariz and right p.p-ring. Then by Proposition \ref{prop1}, $R$ is abelian. Let 
 $f(x)=\sum _{i=0}^{n}a_{i}x^{i}$ and $ g(x)= \sum _{j=0}^{m}b_{j}x^{j} $, $f(x)g(x)=0$. We have
 \begin{eqnarray}
 a_{0}b_{0}=0\\
 a_{0}b_{1}+a_{1}\alpha(b_{0})=0\\
   a_{0}b_{2}+a_{1}\alpha(b_{1})+a_{2}\alpha^{2}(b_{0})=0
 \end{eqnarray}
 By hypothesis there exists idempotent $ e_{i}\in R $ such that $ r(a_{i} ) =e_{i}R$, for all $ i $. Therefore $ b_{0}=e_{0}b_{0} $ and  $ a_{0}e_{0}=0 $. Multiplying (2) by $ e_{0} $ from the right, then 
 $ 0=a_{0}b_{1}e_{0}+a_{1}\alpha(b_{0})e_{0}=a_{0}e_{0}b_{1}+a_{1}\alpha(b_{0})e_{0}=a_{1}\alpha(b_{0})e_{0} $. Hence 
 $ a_{1}\alpha(b_{0})=a_{1}\alpha(b_{0}e_{0})=0 $.
 By (2) 
 $ a_{0}b_{1}=0 $ and so 
 $ b_{1}=e_{0}b_{1} $. Again multiplying (3) from the right by 
 $ e_{0} $, presents
  \begin{align*}
0 &=a_{0}b_{2}e_{0}+a_{1}  \alpha(b_{1}) e_{0}+a_{2}\alpha^{2}(b_{0})e_{0}\\
&=a_{0}b_{2}e_{0}+a_{1}\alpha(b_{1})\alpha(e_{0})+a_{2}\alpha^{2}(b_{0})\alpha^{2}(e_{0})\\
&=a_{0}e_{0}b_{2}+a_{1}\alpha(e_{0}b_{1})+a_{2}\alpha^{2}(e_{0}b_{0}) \\
&=a_{1}\alpha(b_{1}) + a_{2}\alpha^{2}(b_{0})
 \end{align*}
 Multiplying this equation from the right by $ e_{1} $, hence
  \begin{align*}
0 &=a_{1}\alpha(b_{1}) e_{1}+ a_{2}\alpha^{2}(b_{0})e_{1}\\
&=a_{1}e_{1}\alpha(b_{1})+a_{2}\alpha^{2}(b_{0})e_{1}\\
&=a_{2}\alpha(\alpha(b_{0}))\alpha(e_{1})\\
&=a_{2}\alpha(\alpha(b_{0})e_{1})\\
&=a_{2}\alpha(\alpha(b_{0}))\\
&=a_{2}\alpha^{2}(b_{0})
 \end{align*}
continuing this process, we have 
$ a_{i}\alpha^{i}(b_{j})=0 $
for all
$ 0\leqslant i\leqslant n $
and 
$ 0\leqslant j\leqslant m $.
Hence 
$ R $
is $\alpha$-skew Armendariz. This completes the proof.
 \end{proof}
In \cite[Example 2.3]{a01},  it is shown that the hypothesis that $R$ be right p.p-ring is essential in Theorem \ref{THM1} for the endomorphism $\alpha=I_{R}$.

\begin{cor}\label{THM2}
Assume that $\alpha$ is an automorphism of a ring $R$ with $\alpha(e)=e$ for any $e^{2}=e \in R$. If $R$ is a central $\alpha$-skew Armendariz ring, then $R$ is right p.p-ring if and only if $R[x;\alpha]$ is right p.p-ring.
\end{cor}
\begin{proof}
Let $R$ be right p.p-ring. By Theorem \ref{THM1}, $R$ is $\alpha$-skew Armendariz. So the proof is done by
 \cite[Theorem 22]{alpha} . 
Conversely, assume that $R[x;\alpha]$ is a right p.p-ring. Let $a\in R$. By Lemma \ref{Lem1}, there exists an idempotent $e\in R$ such that
 $r_{R[x;\alpha]}(a)=eR[x;\alpha]$. Hence
 $r_{R}(a)=eR$. Therefore $R$ is a p.p-ring.
\end{proof}
Recall that if
$\alpha$
is an endomorphism of a ring $R$, then the map 
$R[x]\longrightarrow R[x]$
defined by 
$\sum _{i=0}^{m}a_{i}x^{i}\mapsto \sum _{i=0}^{m}\alpha(a_{i})x^{i}$
is an endomorphism of the polynomial ring 
$R[x]$
and clearly this map extends 
$\alpha$.
We shall also denote the extended map
$R[x]\longrightarrow R[x]$
by
$\alpha$
and the image of 
$f\in R[x]$
by
$\alpha(f)$.\\
Note that by \cite[Theorem 6]{alpha}, a ring $R$ is $\alpha$-skew Armendariz if and only if $R[x]$ is $\alpha$-skew Armendariz for an endomorphism $\alpha$ with $\alpha^{t}=I_{R}$ for some positive integer $t$. Similarly we have the following result.
 \begin{thm}
 Let $\alpha$ be an endomorphism of a ring $R$ and $\alpha^{t}=I_{R}$ for some positive integer $t$. Then $R$ is central $\alpha$-skew Armendariz if and only if $R[x]$ is central $\alpha$-skew Armendariz. 
 \end{thm}
\begin{proof}
Assume that $R[x]$ is central $\alpha$-skew Armendariz. Then $R$ is central $\alpha$-skew Armendariz az a subring of $R[x]$.
Conversely, assume that $R $ is central $\alpha$-skew Armendariz. Suppose that 
$ p(y)= f_{0}+f_{1}y+\ldots +f_{m}y^{m} $; $q(y)=g_{0}+g_{1}y+\ldots +g_{n}y^{n} $ in $R[x][y;\alpha]$ and
$p(y)q(y)=0$.
We also let $f_{i}=a_{i_{0}}+a_{i_{1}}x+\ldots+a_{i \omega_{i}}x^{\omega_{i}}$ and 
$g_{j}=b_{j_{0}}+b_{j_{1}}x+\ldots+b_{j \nu_{j}}x^{\nu_{j}}$
for each $ 0\leqslant i\leqslant m $
and 
$ 0\leqslant j\leqslant n $, where $a_{i_{0}}, \ldots ,a_{i \omega_{i}}, b_{j_{0}}, \ldots ,a_{j \nu_{j}}\in R$ .
We claim that 
$f_{i}\alpha^{i}(g_{j})\in C(R[x])$, for all  $ 0\leqslant i\leqslant m $
and 
$ 0\leqslant j\leqslant n $.
Take a positive integer $k$ such that 
$k\geqslant  deg (f_{0}(x))+ deg(f_{1}(x))+\cdots +deg(f_{m}(x)) +deg(g_{0}(x))+deg(g_{1}(x))\cdots +deg(g_{n}(x))$,
where the degree is as polynomials in $R[x]$ and the degree of zero polynomial is taken to be $0$.
Since
$p(y)q(y)=0 $ in 
$R[x][y;\alpha]$, we have

\begin{align}\label{4}
\begin{cases}
   f_{0}(x)g_{0}(x)=0\\
f_{0}(x)g_{1}(x)+f_{1}(x)\alpha(g_{0}(x))=0\\
\vdots\\
f_{m}(x)\alpha^{m}(g_{n}(x))=0
\end{cases}
\end{align}

in $R[x]$. Now put
\begin{align}\label{5}
f(x)=f_{0}(x^{t})+f_{1}(x^{t})x^{tk+1}+f_{2}(x^{t})x^{2tk+2}+\cdots +f_{m}(x^{t})x^{mtk+m}
\end{align}
\begin{align*}
g(x)=g_{0}(x^{t})+g_{1}(x^{t})x^{tk+1}+g_{2}(x^{t})x^{2tk+2}+\cdots +g_{n}(x^{t})x^{ntk+n}.
\end{align*}
Note that 
$\alpha^{t}=I_{R}$, Then
\begin{center}
$f(x)g(x)=f_{0}(x^{t})g_{0}(x^{t})+\big( f_{0}(x^{t})g_{1}(x^{t})+f_{1}(x^{t})\alpha(g_{0}(x^{t}))\big) x^{tk+1}+\cdots + f_{m}(x^{t})\alpha^{m}(g_{n}(x^{t}))x^{m+n(tk+1)}$
\end{center}
in $R[x;\alpha]$.  Using 
(\ref{4}) and
$\alpha^{t}=I_{R}$, we have 
$f(x)g(x)=0$ in
$R[x;\alpha]$.
In the other hand, from (\ref{5}) we have 
\begin{center}
$f(x)g(x)=(a_{00}+a_{01}x^{t}+\cdots +a_{0\omega_{0}}x^{\omega_{0} t}+
a_{10}x^{tk+1}+a_{11}x^{tk+t+1}+\cdots +a_{1\omega_{1}}x^{tk+\omega_{1} t+1}
+\cdots +
a_{m0}x^{mtk+m}+a_{m1}x^{mtk+t+m}+\cdots +a_{m\omega_{m}}x^{mtk+\omega_{m} t+m})
(b_{00}+b_{01}x^{t}+\cdots +b_{0\nu_{0}}x^{\nu_{0} t}+
b_{10}x^{tk+1}+b_{11}x^{tk+t+1}+\cdots +b_{1\nu_{1}}x^{tk+\nu_{1} t+1}
+\cdots +
b_{n0}x^{ntk+n}+b_{n1}x^{ntk+t+n}+\cdots +b_{n\omega_{n}}x^{ntk+\nu_{n} t+n})=0$
\end{center}
 Since $R$ is central $\alpha$-skew Armendariz and $\alpha^{t}=I_{R}$, so
 $a_{iu}\alpha^{i}(b_{iv})=a_{iu}\alpha^{itk+ut+i}(b_{jv}) \in C(R)$.
  for each 
 $ 0\leqslant i\leqslant m $,
$ 0\leqslant j\leqslant n $,
and
$u\in  \lbrace 0,1, \ldots ,\omega_{i} \rbrace$,
$v\in  \lbrace 0,1, \ldots ,\nu_{j}\rbrace$.
 Since $C(R)$ is closed under addition, we have
 $f_{i}(x^{t})\alpha^{i}(g_{j}(x^{t}))\in C(R[x])$ 
 for every
  $ 0\leqslant i\leqslant m $ and $ 0\leqslant j\leqslant n $.
 Now it is easy to see that
 $f_{i}(x)\alpha^{i}(g_{j}(x))\in C(R[x])$, 
 and hence $R[x]$ is central $\alpha$-skew Armendariz.
\end{proof}
Let 
$ R $ be a ring. 
For any integer 
$n\geqslant2$, consider the ring 
$M_{n}(R)$ of
$n\times n $ matrices and the ring $T_{n}(R)$ of $n\times n$ triangular matrices  over a ring $R$. 
The $n\times n$ identity matrix is
denoted by $ I_{n} $. For $n\geqslant 2$, let 
$\lbrace e_{i,j| 1\leqslant i,j\leqslant n}\rbrace$
be the set of the matrix units.
Let $\alpha:R\longrightarrow R$ be a ring endomorphism with $\alpha(1)=1$. For any
$A=(a_{i,j})\in M_{n}(R)$,
 we denote
 $\bar{\alpha}:  M_{n}(R)\longrightarrow  M_{n}(R)$
by
$\bar{\alpha}((a_{i,j})_{n\times n})=(\alpha(a_{i,j}))_{n\times n}$,
 and so $\bar{\alpha}$ is
a ring endomorphism of the ring $M_{n}(R)$.\\
The rings 
$M_{n}(R)$ and  $T_{n}(R)$ 
 contain non-central idempotents. Therefore they are not abelian. By Proposition \ref{prop1}, these rings are not central $I_{R}$-skew Armendariz.

Now we introduce a notation for some subring of $T_{n}(R)$ that will be central $\bar{\alpha}$-skew Armendariz. \\ 
Given a ring $R$ and $(R,R)$-bimodule $M$, the trivial extention of $R$ by $M$ is the ring $T(R,M)=R\oplus M$ with the usual addition and the following multiplication:
\begin{center}
$(r_{1},m_{1})(r_{2},m_{2})=(r_{1}r_{2},r_{1}m_{2}+m_{1}r_{2})$.
\end{center}
This is isomorphic to the ring of all matrices 
$\bigl( \begin{smallmatrix}
  r&m\\ 0&r
\end{smallmatrix} \bigr)$, 
where 
$r\in R$
and 
$m\in M$ and usual matrix operations are used.\\
For an endomorphism $\alpha$ of a ring $R$ and the trivial extention $T(R,R)$ of $R$,
$\bar{\alpha}:T(R,R)\longrightarrow T(R,R)$
defined by 
$\bar{\alpha}((a,b))=(\alpha(a),\alpha(b))$ is an endomorphism of 
$T(R,R)$.
Since 
$T(R,0)$
is isomorphic to $R$, we can describe the restriction of $\bar{\alpha}$ by $T(R,0)$ to 
$\alpha$.\\
If $R$ is an $\alpha$-rigid ring (i.e., $R[x;\alpha]$ is reduced) by \cite[Proposition 15]{alpha},  
$T(R,R)$
is 
$\bar{\alpha}$-skew Armendariz and so it is central $\bar{\alpha}$-skew Armendariz. But 
$T(R,R)$ need not to be $\bar{\alpha}$-rigid.\\
 It can be asked that if 
$T(R,R)$
is a central $\bar{\alpha}$-skew Armendariz ring, then $R $ is $\alpha$-rigid ring. But this is not the case.
\begin{examp}
Let $R=\mathbb{Z}_{4}$, where $\mathbb{Z}_{4}$ is the ring of integers modulo $4$. Then  $T(R,R)$ is a commutative ring and hence for $\alpha=I_{R}$ is central $\bar{\alpha}$-skew Armendariz. But  $R[x]$ is not reduced and so $R$ is not rigid by \cite[Proposition 3]{alpha}.

\end{examp}
Let $\alpha$ be an monomorphism of a ring $R$ and $R$ be a
reduced $\alpha$-skew Armendariz ring, then
\begin{center}
$ S= \Bigg\{\Bigl(\begin{smallmatrix}
  a & b & c \\
  0 & a & d \\
  0 &0 & a
\end{smallmatrix}\Bigr)| a,b,c,d \in R\Bigg\}$
\end{center}
is $\bar{\alpha}$-skew Armendariz by \cite[Corollary 5]{a4}, and so it is central $\bar{\alpha}$-skew Armendariz.
On the other hand by \cite[Proposition 17]{alpha}, if $\alpha$ is an endomorphism of a ring $R$ and $R$ is an $\alpha$-rigid ring then $S$ is $\bar{\alpha}$-skew Armendariz and so it is central $\alpha$-skew Armendariz.\\

For an ideal $I$ of $R$, if $\alpha(I) \subseteq I$ then $\bar{\alpha}:R/I\longrightarrow R/I$ defined by $\bar{\alpha}(a+I)=\alpha(a)+I$
is an endomorphism of a factor ring $R/I$.
The homomorphic image of a central $\alpha$-skew Armendariz ring need not be central $\alpha$-skew Armendariz. 
But by in \cite[Proposition 9]{alpha}, if for any $a\in R$, $a\alpha(a)\in I$ implies $a\in I$  then the factor ring $R/I$ is
$\bar{\alpha}$-skew Armendariz and so is central $\bar{\alpha}$-skew Armendariz.\\
Recall that a ring R is $\alpha$-compatible if for each $ a,b \in R$, $ab=0$ if and only if $a\alpha(b)=0$. Clearly, this may only happen when the endomorphism $\alpha$ is injective.
Note the following result.
\begin{thm}\label{thme14}
Let $\alpha$ be an endomorphism of a ring $R$ with $ \alpha(1)=1 $, $R$ be an $\alpha$-compatible ring and $I$ be an ideal of $R$ with
 $\alpha(I)\subseteq I$.
 If $I$ is reduced as a ring and $R/I $ is central $\alpha$-skew Armendariz ring, then for 
 $f(x)=\sum_{i=0}^{n}x^{i}$ and
$g(x)=\sum_{j=0}^{m}x^{j}\in R[x;\alpha]$, if
$f(x)g(x)=0$ and $a_{0}\in U(R) $, then $ R $ is central $ \alpha $-skew Armendariz.
\end{thm}
\begin{proof}
Let 
$a,b\in R$. Since $R$ is $\alpha$-compatible
$ab=0$ implies that  for any 
$n\in \mathbb{N}$,
$a\alpha^{n}(b)=0$.
Then 
$(\alpha^{n}(b)Ia)^{2}=0$. Since 
$\alpha^{n}(b)Ia\subseteq I$ and
$I$
is reduced,
$\alpha^{n}(b)Ia=0$.
Also,
$(aI\alpha^{n}(b))^{3}\subseteq (aI\alpha^{n}(b))(I)(aI\alpha^{n}(b))=0 $.
Therefore 
$aI\alpha^{n}(b)=0$. Assume
$f(x)=a_{0}+a_{1}x+\cdots +a_{n}x^{n}$,
$g(x)=b_{0}+b_{1}x+\cdots +b_{m}x^{m}\in R[x;\alpha]$ and 
$f(x)g(x)=0$. Then
\begin{eqnarray}
a_{0}b_{0}=0
\label{eq1}
\\
a_{0}b_{1}+a_{1}\alpha(b_{0})=0
\label{eq2}
\\
a_{0}b_{2}+a_{1}\alpha(b_{1})+a_{2}\alpha^{2}(b_{0})=0
\label{eq3}
\end{eqnarray}
We first show that for any 
$a_{i}\alpha^{i}(b_{j})$,
$a_{i}I\alpha^{i}(b_{j})=\alpha^{i}(b_{j})Ia_{i}=0$.
Multiply (\ref{eq2}) from the right by 
$I\alpha^{n}(b_{0})$, we have
$a_{1}\alpha(b_{0})I\alpha^{n}(b_{0})=0$, since
$a_{0}b_{1}I\alpha^{n}(b_{0})=0$.
Then
$(\alpha^{n}(b_{0})Ia_{1})^{3}\subseteq \alpha^{n}(b_{0})I(a_{1}\alpha^{n}(b_{0})Ia_{1}\alpha^{n}(b_{0}))Ia_{1}=0$.
Hence
$\alpha^{n}(b_{0})Ia_{1}=0$. This implies 
$a_{1}I\alpha^{n}(b_{0})=0$.
Multiply (\ref{eq2}) from the left by
$a_{0}I$,
we have 
$a_{0}Ia_{0}b_{1}+a_{0}Ia_{1}\alpha(b_{0})=0$ and so 
$a_{0}Ia_{0}b_{1}=0$. Thus 
$(b_{1}Ia_{0})^{3}=0$ and 
$b_{1}Ia_{0}=0$.
Now multiply (\ref{eq3}) from right by 
$I\alpha^{n}(b_{0})$. 
Then 
$a_{2}\alpha^{2}(b_{0})I\alpha^{n}(b_{0})=0$ and 
$(\alpha^{2}(b_{0})Ia_{2})^{3}=0$. So
$\alpha^{2}(b_{0})Ia_{2}=0$ and 
$a_{2}I\alpha^{2}b_{0}=0$ and 
$a_{2}I\alpha^{2}(b_{0})=0$.
 Now from \ref{eq3} we have $a_{0}b_{2}I+a_{1}\alpha(b_{1})I+a_{2}\alpha^{2}(b_{0})I=0$.
 Since
 $a_{0}b_{1}+a_{1}\alpha(b_{0})=0$ and
$\alpha(a_{0})=a_{0}$
 the square of 
 $a_{0}b_{2}I$ and
 $a_{2}\alpha^{2}(b_{0})I$ are zero, 
 $a_{0}b_{2}I=a_{2}\alpha^{2}(b_{0})I=0$ so 
 $a_{1}\alpha(b_{1})I=0$. Then 
 $(\alpha(b_{1})Ia_{1})^{2}=0$ and $\alpha(b_{1})Ia_{1}=0$. So $a_{1}I\alpha(b_{1})=0$. Continuing in this way we have 
 $a_{i}I\alpha^{i}(b_{j})=\alpha^{i}(b_{j})Ia_{i}=0$.
Since 
$R/I$ is central $\alpha$-skew Armendariz, it follows that
$a_{i}\alpha^{i}(b_{j})\in C(R/I)$. So 
$a_{i}\alpha^{i}(b_{j})r-ra_{i}\alpha^{i}(b_{j})\in I$
for any
$r\in R$.
Now from the above results, we have 
$(a_{i}\alpha^{i}(b_{j})r-ra_{i}\alpha^{i}(b_{j}))I(a_{i}\alpha^{i}(b_{j})r-ra_{i}\alpha^{i}(b_{j}))=0$.Then 
$a_{i}\alpha^{i}(b_{j})r=ra_{i}\alpha^{i}(b_{j})$
for all
$r\in R$.
Hence 
$a_{i}\alpha^{i}(b_{j})$ is central for all
$i$ and
$j$. This completes the proof.
\end{proof}
Note that in Theorem \ref{thme14}, if $R$ is an $\alpha$-rigid ring instead of $\alpha$-compatible, then by \cite[Proposition 8]{alpha} $R$ should be central $\alpha$-skew Armendariz.
 The following example, shows that there exists a non-identity endomorphism $\alpha$ of a ring $R$ such that $R/I$ 
is central $\bar{\alpha}$-skew Armendariz and as a ring $I$ is central $\alpha$-skew Armendariz for any nonzero proper ideal $I$ of $R$, but $R$ is not central $\alpha$-skew Armendariz .
\begin{examp}
Let $F$ be a field and consider a ring
$R=\bigl( \begin{smallmatrix}
  F&F\\ 0&F
\end{smallmatrix} \bigr)$
and an endomorphism $\alpha$ of $R$ defined by
$\alpha \Big( \bigl( \begin{smallmatrix}
  a&b\\ 0&c
\end{smallmatrix} \bigr)\Big) =\bigl( \begin{smallmatrix}
  a&-b\\ 0&c
\end{smallmatrix} \bigr)$
For
$f(x)= \bigl( \begin{smallmatrix}
  1&0\\ 0&0
\end{smallmatrix} \bigr) +\bigl( \begin{smallmatrix}
  1&1\\ 0&0
\end{smallmatrix} \bigr)x ,\quad g(x)= \bigl( \begin{smallmatrix}
  0&0\\ 0&-1
\end{smallmatrix} \bigr) +\bigl( \begin{smallmatrix}
  0&1\\ 0&1
\end{smallmatrix} \bigr)x \in R[x;\alpha]$
we have $f(x)g(x)=0$, but
$ \bigl( \begin{smallmatrix}
  1&1\\ 0&0
\end{smallmatrix} \bigr) \alpha \Big(\bigl( \begin{smallmatrix}
  0&0\\ 0&-1
\end{smallmatrix} \bigr)\Big)=\bigl( \begin{smallmatrix}
  0&-1\\ 0&0
\end{smallmatrix} \bigr) \notin C(R)$. 
Thus $R$ is not central $\bar{\alpha}$-skew Armendariz. But by \cite[Example 12]{alpha},
for 
$I=\bigl( \begin{smallmatrix}
  F&F\\ 0&0
\end{smallmatrix} \bigr) $, $\bigl( \begin{smallmatrix}
  0&F\\ 0&F
\end{smallmatrix} \bigr)$and
$\bigl( \begin{smallmatrix}
  0&F\\ 0&0
\end{smallmatrix} \bigr)$
$R/I$ and $I$ are $\bar{\alpha}$-skew Armendariz and$\alpha $-skew Armendariz respectively and so $R/I$ and $I$ are central $\bar{\alpha}$-skew Armendariz and central $\alpha$-skew Armendariz respectively.
\end{examp}
However, we have the following result.
\begin{thm}
Let $\alpha$ be a monomorphism of a ring $R$, and $\alpha(1)=1$ where $1$ denotes the identity of $R$. If $R$ is $\alpha$-rigid, then a factor ring
 $\frac{R[x]}{\langle x^{2}\rangle}$ 
 is central $\bar{\alpha}$-skew Armendariz, where
  $\langle x^{2}\rangle$
 is an ideal of $R[x]$ generated by 
 $x^{2}$. The converse holds if $R$ is prime.
\end{thm} 
\begin{proof}
 Let
 $R$
  be $\alpha$-rigid. Then by \cite[Proposition 8]{alpha},
 $\frac{R[x]}{\langle x^{2}\rangle}$ is 
 $\bar{\alpha}$-skew Armendariz, and so is central 
 $\bar{\alpha}$-skew Armendariz.
Conversely, assume that 
 $\frac{R[x]}{\langle x^{2}\rangle}$ 
 is central $\alpha$-skew Armendariz. Let 
 $r\in R$
 with 
 $\alpha(r)r=0$.
 Then 
\begin{center}
 $(\alpha(r)-\bar{x}y)(r+\bar{x}y)=\alpha(r)r+(\alpha(r)\bar{x}-\bar{x}\alpha(r))y-\alpha(1)\bar{x}^{2}y^{2}=\bar{0}$
\end{center}
Since
$\alpha(r)\bar{x}=\bar{x}\alpha(r)$ in
 $\frac{R[x]}{\langle x^{2}\rangle}[y;\alpha]$ ,
 where 
 $\bar{x}=x+\langle x^{2}\rangle\in\frac{R[x]}{\langle x^{2}\rangle} $.
 Since 
 $\frac{R[x]}{\langle x^{2}\rangle}$ is central $\alpha$-skew Armendariz,
 $\alpha(r)\bar{x}\in C(\frac{R[x]}{\langle x^{2}\rangle})$  
 and so 
 $\alpha(r)\bar{x}\bar{a}=\bar{a}\alpha(r)\bar{x}$
 for any
 $a\in R$.
 Then 
 $\alpha(r)a=a\alpha(r)$, hence 
 $\alpha(r)Rr=0$.
 Since $R$ is prime and $\alpha$ is injective, $r=0$.
 Therefore $R$ is $\alpha$-rigid.
 \end{proof}
  Let 
 $\alpha $ be an automorphism of a ring $R$. Suppose that there exists the classical right quotient ring $Q(R)$ of $R$. Then for  any $ab^{-1}\in Q(R) $ where 
 $a,b\in R$ with $b$ regular, the induced map 
 $\bar{\alpha}:Q(R)\longrightarrow Q(R)$
 defined by 
 $\bar{\alpha}(ab^{-1})=\alpha(a)\alpha(b)^{-1}$ is also an automorphism.
 Note that $R$ is $\alpha$-rigid if and only if $Q(R)$ is $\bar{\alpha}$-rigid. Hence if $R$ is $\alpha$-rigid, then $Q(R)$
 is $\bar{\alpha}$-skew Armendariz, and so is central $\bar{\alpha}$-skew Armendariz.\\
 Let
 $S$ denote a multiplicatively closed subset of a ring $R$ consisting of central regular elements. And let $RS^{-1}$ be the localization of $R$ at $S$. 
\begin{prop}\label{prop17}
Let $\alpha$ be an automorphism of a ring $R$.  Then $R$ is central $\alpha$-skew Armendariz, if and only if $RS^{-1}$ is central $\bar{\alpha}$-skew Armendariz.
\end{prop}
\begin{proof}
Suppose that $R$ is central $\alpha$-skew Armendariz ring. Let 
$f(x)= \sum_{i=0}^{n}(a_{i}/s_{i})x^{i}$,
$g(x)= \sum_{j=0}^{m}(b_{j}/d_{j})x^{j} \in  RS^{-1}[x;\alpha]$ and
$f(x)g(x)=0$.
Let
$a_{i}s_{i}^{-1}=c^{-1}a_{i}^{\prime}$
and
$b_{j}d_{j}^{-1}=d^{-1}b_{j}^{\prime}$
with $c, d$
regular elements in $R$. Then we have
$(a_{0}^{\prime}+\cdots+ a_{n}^{\prime}x^{n})d^{-1}(b_{0}^{\prime}+\cdots+ b_{m}^{\prime}x^{m})=0$.
We know that for each element 
$f(x) \in RS^{-1}[x;\bar{\alpha}]$
there exists a regular element $c\in R$
such that $f(x)=h(x)c^{-1}$, for some
$h(x)\in R[x;\alpha]$, or equivalently, 
$f(x)c\in R[x;\alpha]$. Therefore there exist a regular element $e $ in $R$ and
$(b_{0}^{\prime\prime}+\cdots+ b_{t}^{\prime\prime}x^{t})\in R[x;\alpha]$,
 such that
$d^{-1}(b_{0}^{\prime}+\cdots+ b_{m}^{\prime}x^{m})=(b_{0}^{\prime\prime}+\cdots+ b_{t}^{\prime\prime}x^{t})e^{-1}$.
Hence 
$(a_{0}^{\prime}+\cdots+ a_{n}^{\prime}x^{n})(b_{0}^{\prime\prime}+\cdots+ b_{t}^{\prime\prime}x^{t})=0$.
Since $R$ is central $\alpha$-skew Armendariz, 
$a_{i}^{\prime}\alpha^{i}(b_{j}^{\prime\prime})\in C(R)$ for all $i$ and $j$. Therefore 
 $ca_{i}s_{i}^{-1}\alpha^{i}(b_{j}d_{j}^{-1}e^{-1})\in C(R)$. Since $c$ and $e$ are regular element of $R$, 
 $a_{i}s_{i}^{-1}\alpha(b_{j}d_{j}^{-1}) $ are central in $R$ for all $i$ and $j$.
 Conversely, assume that $RS^{-1}$ is central $\alpha$-skew Armendariz ring. Since subrings of central $\alpha$-skew Armendariz rings are central $\alpha$-skew Armendariz, then $R$ is central $\alpha$-skew Armendariz.
\end{proof}
\begin{cor}
Let $R$ be a ring and $\alpha$ an automorphism of $R$, such that $\alpha^{t}=I_{R}$ for some positive integer $t$, then the following are equivalent:
\item[(1)] $R$ is central $\alpha$-skew Aremendariz.
\item[(2)] $R[x] $ is central $\alpha$-skew Aremendariz.
\item[(3)]$R[x,x^{-1}]$ is central $\alpha$-skew Aremendariz.
\end{cor}
\begin{proof}
Let $S=\lbrace 1,x,x^{2},x^{3}, x^{4},\cdots \rbrace$. Then $S$ is a multiplicatively closed subset of $R[x]$ consisting of central regular elements. Then the proof follows from Proposition \ref{prop17}.
\end{proof}


\end{document}